\documentclass{article}
\usepackage{subcaption}
\usepackage{graphicx} % Required for inserting images

\usepackage{geometry}
\usepackage{tikz}
\usetikzlibrary{automata,arrows,positioning,calc}

\usepackage{tikz}
\usepackage{amssymb}
\usetikzlibrary{automata, positioning, arrows}
\usepackage{amsthm}
\newtheorem{theorem}{Theorem}[section]
\newtheorem{corollary}[theorem]{Corollary} %[section]
\newtheorem{lemma}[theorem]{Lemma} %[section]
\theoremstyle{definition}
\newtheorem{definition}[theorem]{Definition} %[section]

\newtheorem{prop}[theorem]{Proposition} %[section]
\newtheorem{remark}[theorem]{Remark} %[section]
\usepackage{amsmath}

\usepackage{geometry}
\usepackage{comment}

\newcommand{\sP}{\mathcal P}
\newcommand{\sQ}{\mathcal Q}
\newcommand{\sR}{\mathcal R}

\title{Optimal Stopping of a Brownian Excursion and an $\alpha-$dimensional Bessel Bridge}
\author{David Hobson, Jingfei Liu}
\date{\today}

\begin{document}

\maketitle

\begin{abstract}

We study the optimal stopping of an $\alpha$-dimensional Bessel bridge for the payoff $\phi(x)=x^n$, where
$\alpha,n>0$. As a
special case we consider the Brownian excursion with the identity function as the payoff ($\alpha=3,n=1$). For
the Brownian
excursion we can give an explicit solution but in the general case we provide a complete solution via a power
series expansion.

\end{abstract}

\section{Introduction}
In this paper, we start by studying the optimal stopping problem for a Brownian excursion. Let $X =
\{X_{t}\}_{0\leq t\leq 1}$ be a
scaled Brownian excursion process started at zero, or equivalently a Brownian motion conditioned to be positive on time-period $(0,1)$ and to return
to zero for the first time at time 1.
We find $\sup_{\tau \leq 1} \mathbb{E}[X_\tau]$ and the optimal stopping time $\tau^*$.

A Brownian excursion process is a 3-dimensional Bessel bridge. In the second part of the paper, we
generalize the above
optimal stopping problem to consider an $\alpha-$dimensional Bessel bridge as the underlying stochastic process,
where $\alpha > 0$,
and a payoff function which is a positive power, $\phi(x)=x^n$ where $n>0$. We look for a value function for the
process started at
$(t,x) \in [0,1) \times [0,\infty)$
\begin{equation*}
    V(t,x)=\sup_{t\leq\tau\leq 1}\mathbb{E}_{t,x}(X_{\tau}^{n}),
\end{equation*}
and the associated optimal stopping time.

Related optimal stopping problems have been studied extensively in the literature. Shepp~\cite{shepp1969explicit}
first introduced the optimal stopping problem where the underlying stochastic process is a Brownian bridge. Shepp derived an
explicit solution, and solved the problem by transforming the Brownian bridge into a time-changed Brownian motion. Ekstrom and
Wanntorp~\cite{Ekstrom_Wanntorp} solved the same problem using a more direct approach and extended the result to include different
payoff functions. Later, Gorgens~\cite{görgens2014optimalstoppingalphabrownianbridge} extended the analysis to an $\alpha-$Brownian bridge (i.e a diffusion process $X$ satisfying the SDE $dX_t = dB_t - \frac{\alpha X_t}{1-t}dt$). The study of optimal stopping with Brownian bridges can be applied in finance, as Brownian bridges capture the stock pinning effect. With this in mind Baurdoux et al.\cite{Baurdoux_Chen_Surya_Yamazaki_2015} extended the original optimal stopping problem to an optimal double stopping problem driven by a Brownian bridge; their goal is to find a pair of stopping times that maximizes the expected spread between the two payoffs. Their study provides a solution for situations where one aims to buy and sell an asset to maximize the spread before the price converges to a fixed value. D'Auria et al.\cite{d2020discounted} studied the discounted optimal stopping of a Brownian bridge, as well as its application in American options under pinning. De Angelis and Milazzo~\cite{de2020optimal} studied the problem for the exponential of the Brownian bridge. Finally (for our review), Glover
solved the optimal stopping problem
with an unknown pinning time using a Bayesian approach~\cite{glover2022optimally}.

In this paper, we consider an $\alpha-$dimensional Bessel bridge as the underlying stochastic process for the
optimal stopping
problem, and we choose the payoff function to be $X^{n}$. (Here $\alpha$ and $n$ are positive, but not necessarily integers.) We find an explicit solution for the case where $\alpha=3$
and $n=1$, and
closed form solutions for some other paramater values; however in the general case the solution is presented as a
power series
solution. To the best of our knowledge, this is the first time that Bessel bridges have been used as the driving
stochastic processes
for an optimal stopping problem.

The method we used to find the value function and the optimal boundary was introduced by Ekstrom and
Wanntorp\cite{Ekstrom_Wanntorp}.
We allow for an arbitrary starting point, and formulate a free boundary problem for the value function as well as
the optimal
stopping boundary. We then find a candidate solution for this free boundary problem. Finally, we verify that this
candidate
solution is indeed the true solution.

\section{Optimal stopping of a Brownian excursion}
\label{sec:BE}
Let $W=\{W_{t}\}_{t\geq 0}$ be a standard Brownian motion.
\begin{definition}\label{def_BE}
    A non-negative continuous process $X=\{X_{t}\}_{0\leq t \leq 1}$ is a (scaled) Brownian excursion process if
    it satisfies the
    SDE,
    \begin{equation}
    \label{eq:defBE}
        dX_{t}=\left(\frac{1}{X_{t}}-\frac{X_{t}}{1-t}\right)dt + dW_{t},  \qquad   0<t<1,
    \end{equation}
    subject to intial condition $X_{0}=0$. See Pitman~\cite{pitman1999sde} for more details.
\end{definition}

Note that the $\frac{1}{X_t}$ term in the drift prevents the process from hitting zero before time 1, and the
$\frac{-X_t}{1-t}$ term makes sure that it returns to zero at time $1$.
In this section, we are interested in the optimal stopping problem for $X=\{X_{s}\}_{t\leq s\leq 1}$, a Brownian
excursion started at $(t,X_{t})=(t,x)$, with the identity payoff. The associated value function is
\begin{equation}\label{value_def}
    V(t,x)=\sup_{t\leq \tau\leq 1}\mathbb{E}_{t,x}[X_{\tau}],
\end{equation}
and the goal is to find the optimal stopping time $\tau^{*}$ for which the supremum in \eqref{value_def} is attained. 
Given the Markovian structure of the problem we expect that the optimal stopping rule will be of threshold form:
$\tau^* = \tau_c$
where
\begin{equation*}
    \tau_c = \inf\{s\geq t: X_{s}\geq c(s)\},
\end{equation*}
where $c( \cdot )$ denotes the optimal stopping boundary. The region $\{ (t,x) \in (0,1) \times [0,c(t)) \}$ is called the continuation region, and $\{ (t,x) \in (0,1) \times [c(t),\infty)\}$ is called the
stopping region.

When $X_{t}$ is in the stopping region, i.e. $X_{t}\geq c(t)$, we expect $V(t,X_{t})=
X_{t}$. Thus, we only
need to find $V(t, X_{t})$ when $X_{t}$ is in the continuation region. Assuming that $V$ is sufficiently smooth,
we deduce from
It\^{o}'s formula and
\eqref{eq:defBE} that
\begin{equation*}
    \begin{split}
        dV(t, X_{t})=\left[V_{t}(t,X_{t}) +V_{x}(t,X_{t})\left(\frac{1}{X_{t}}-\frac{X_{t}}{1-t}\right)
        +\frac{1}{2}V_{xx}(t,
        X_{t})\right]dt+V_{x}(t,X_{t})dW_{t}.
    \end{split}
\end{equation*}
Since the value function evolves as a martingale in the continuation region $x<c(t)$ we obtain the HJB equation,
\begin{equation}
    V_{t}(t,x) +\left(\frac{1}{x}-\frac{x}{1-t}\right)V_{x}(t,x) + \frac{1}{2}V_{xx}(t, x) = 0  \label{PDE}
\end{equation}
with boundary conditions
\begin{align}
    V(t,c(t))&= c(t), \label{boundary_cond_1}\\
    V_{x}(t,c(t))&= 1, \label{boundary_cond_2}\\
    V(t,0)&<\infty. \label{boundary_cond_3}
\end{align}
Here \eqref{boundary_cond_1} and \eqref{boundary_cond_2} are the value matching condition and smooth fit condition
respectively and
\eqref{boundary_cond_3} ensures that the value function is finite when $x=0$.

We make the ansatz, $c(t)=C\sqrt{1-t}$ together with
\begin{equation}
\label{eq:Vansatz}
    V(t,x) = \sqrt{1-t}f\left(\frac{x}{\sqrt{1-t}}\right),
\end{equation}
where $C$ is a constant and $f:\mathbb{R}_+ \to \mathbb{R}_+$ a function to be determined. In this way, setting
$y=\frac{x}{\sqrt{1-t}}$, the problem given in \eqref{PDE}-\eqref{boundary_cond_3} is transformed into
\begin{equation}
\label{ODE}
    f''(y)+ \left( \frac{2}{y}-y \right)f'(y)-f(y)=0
\end{equation}
on $y<C$, with boundary conditions
\begin{align}
    f(C)&=C , \label{bond_cond_ode_1}\\
    f'(C)&=1 , \label{bond_cond_ode_2}
\end{align}
and such that $f(0)<\infty$. %\label{bond_cond_ode_3}
It is easily seen that $\frac{1}{y}$ is a solution to \eqref{ODE} and then that a second solution is given by
$\frac{1}{y}\int_{0}^{y}e^{\frac{t^{2}}{2}}dt$. Then the general solution to the ODE in \eqref{ODE} is
\begin{equation}
\label{eq:fdef}
    f(y)= \frac{A}{y}+ \frac{B}{y}\int_{0}^{y}e^{\frac{t^{2}}{2}}dt.
\end{equation}
The finiteness condition at zero implies $A=0$. From \eqref{bond_cond_ode_1} and
\eqref{bond_cond_ode_2} we find that coefficient $B$ satisfies
\begin{equation}
\label{eq:Bvalue}
    B = \frac{C^{2}}{\int_{0}^{C}e^{\frac{t^{2}}{2}}dt},
\end{equation}
where $C$ solves
\begin{equation}
\label{eq:Cdef}
    C = 2 e^{-\frac{C^2}{2}} \int_0^C e^{t^2/2} dt.
\end{equation}
Indeed $C$ is the the unique solution of \eqref{eq:Cdef} in $(0,\infty)$. To see this define $h:[0,\infty) \to \mathbb{R}$ by $h(c)= 2\int_{0}^{c}e^{\frac{t^{2}}{2}}dt-ce^{\frac{c^{2}}{2}}$. Note that $h'(c) = e^{\frac{c^{2}}{2}}(1-c^2)$ so that $h$ is increasing on $(0,1)$ and decreasing on $(1,\infty)$. Uniqueness of a positive root of $h$ in $(1,\infty)$ (and the fact the root is in $(1,2)$) follows from the fact that $h(0)=0$, $h(1)= 2\int_{0}^{1}e^{\frac{t^{2}}{2}}dt-\sqrt{e}> 2 - \sqrt{e}>0$, and, from the convexity of $e^{t^2/2}$, $h(2) < 2 \frac{1 +e^2}{2} - 2e^2 < 0$. Numerical results show that the root is given by $C = 1.50339538$.

The next result is technical and the proof is postponed until the appendix.

\begin{lemma}\label{lemma_f}
    Let $f:[0,C] \to [0,\infty)$ be given by \eqref{eq:fdef} with $A=0$ and $B$ given by \eqref{eq:Bvalue}.

    Then $f(0)=B$ and $f'(0)=0$. Moreover, $f$ is convex with $f(y)\geq y$ for all $y \in [0,C]$. 
\end{lemma}

Putting together \eqref{eq:Vansatz} with \eqref{eq:fdef} for the choice $A=0$ and $B$ given by \eqref{eq:Bvalue}
gives a candidate value function $V^{*}(t,x)$
\begin{equation}\label{candidate_value_func}
    V^{*}(t,x)=
    \begin{cases}
    C^{2}\left(\int_{0}^{C}e^{\frac{t^{2}}{2}}dt\right)^{-1}\left(\frac{1-t}{x}\right)
    \int_{0}^{\frac{x}{\sqrt{1-t}}}e^{\frac{t^{2}}{2}} dt
    &\text{if } x<c(t) \\
      x   &\text{if } x\geq c(t),
    \end{cases}
\end{equation}
where $c(t)= C\sqrt{1-t}$ and $C$ solves \eqref{eq:Cdef}. The next step is to show that the candidate value function $V^{*}(t,x)$ is indeed the
true solution to the
optimal stopping problem introduced in \eqref{value_def}.

\begin{theorem}
\label{thm:BE}
    The candidate value function $V^{*}(t,x)$ given in \eqref{candidate_value_func} coincides with the value
    function defined in
    \eqref{value_def}. Moreover, the associated stopping time $\tau^*$ given by
    \begin{equation*}
        \tau^{*}=\inf\{s\geq t: X_{s}\geq c(s)\},
    \end{equation*}
    where $c(s)=C\sqrt{1-s}$ and $C$ is the unique solution to \eqref{eq:Cdef}, is optimal.
\end{theorem}
\begin{proof}
   Fix $t_0 \in [0,1)$.
   We show that $V^{*}(t_0,x)=V(t_0,x)$ by first proving $V^{*}(t_0,x)\geq V(t_0,x)$, and then proving the reverse inequality.

   For $t \geq t_0$ let $M = \{ M_t \}_{t \geq t_0}$ be given by $M_{t}=V^{*}(t,X_{t})$. 
   It\^{o}'s formula gives us
   \begin{equation*}
       dM_{t}=V^{*}_{t}(t,X_{t})dt+V_{x}^{*}(t,X_{t})dX_{t}+\frac{1}{2}V_{xx}^{*}(t,X_{t})d\langle X\rangle_{t}.
   \end{equation*}
   Therefore, recalling \eqref{eq:defBE},
\begin{equation*}
    dM_{t}=\left[V_{t}^{*}(t,X_{t})+\left(\frac{1}{X_{t}}-\frac{X_{t}}{1-t}\right)V_{x}^{*}(t,X_{t})
    +\frac{1}{2}V_{xx}^{*}(t,X_{t})\right]dt+V_{x}^{*}(t,X_{t})dW_{t},
\end{equation*}
and using \eqref{eq:Vansatz} which converts \eqref{PDE} into \eqref{ODE} and vice-versa,
since $f$ solves \eqref{ODE} when $X$ is in the continuation region, i.e. $X_{t}<c(t)$, the drift term vanishes.
In contrast, when $X$ is in the stopping region, i.e. $X_t\geq c(t)$, we have that $V^{*}(t,x)=x$, and the drift
term simplifies. Putting the cases together,
\begin{equation}\label{supermartingale_eq}
    dM_{t}= g(t,X_t)\mathbb{I}\{X_{t}\geq c(t)\}dt +V_{x}^{*}(t,X_{t})dW_{t}
\end{equation}
where
\begin{equation*}
    g(t,x)= x^{-1} \left( 1 -\frac{x^2}{1-t} \right).
\end{equation*}
Note that for each $t \in [0,1)$, $g(t,\cdot)$ is decreasing in $x$ and therefore, for every $x\geq c(t)$,
$g(t,x)\leq g(t,c(t))$.
Moreover,
since $C>1$,
\begin{equation*}
    g(t,c(t))= g(t,C\sqrt{1-t})=\left(\frac{1}{C}-C\right)\sqrt{1-t}<0.
\end{equation*}
It follows that $\int_{t_0}^\cdot \left(\frac{1}{X_{t}}-\frac{X_{t}}{1-t}\right)\mathbb{I}\{X_{t}\geq c(t)\}dt$ is a
decreasing process.

The process $\int_{t_0}^\cdot V^{*}_{x}dW_{t}$ is a local martingale: we want to argue that $\int_{t_0}^{\cdot}
V^{*}_{x}(s,X_s)dW_{s}$ is a martingale. This will follow if $V^*_x$ is bounded, and in the stopping region this is immediate since $V^*_x \equiv 1$.
In the continuation region we have $V^{*}(t,x) = \sqrt{1-t}f\left(\frac{x}{\sqrt{1-t}}\right)$ where
$    f(y)= \frac{B}{y}\int_{0}^{y}e^{\frac{t^{2}}{2}}dt $
and $B$ and $C$ are given in \eqref{eq:Bvalue} and \eqref{eq:Cdef}. Then, $V^{*}_{x}(t,x) = f' \left(\frac{x}{\sqrt{1-t}}\right)$. From Lemma~\ref{lemma_f} we conclude that $f$ is a convex function, and thus $f' \in [f'(0),f'(C)]=[0,1]$ on $[0,C]$. 
Hence $V_{x}^{*}=f'$ is bounded, and $\int_{t_0}^\cdot V_{x}^{*}(s,x)dW_{s}$ is a martingale.

Putting together the results of the last paragraph, $M=\{M_{t}\}_{t\geq t_0}$ is a supermartingale.
Moreover, $V^*(t,x) \geq x$, this follows by definition on $x \geq c(t)$ and from Lemma~\ref{lemma_f} on the continuation region. 
Then $\mathbb{E}_{t_0,x}[X_\tau] \leq \mathbb{E}_{t_0,x}[V^*(\tau,X_\tau)]$. Further, by the Optional Stopping Theorem,
$\mathbb{E}_{t_0,x}[V^{*}(\tau,X_{\tau})]\leq V^{*}(t_0,x)$ and then for any stopping time $\tau$ taking values in $[t_0,1]$, 
\( %\begin{equation*}
    \mathbb{E}_{t_0,x}[X_{\tau}]\leq \mathbb{E}_{t_0,x}[V^{*}(\tau,X_{\tau})]\leq V^{*}(t_0,x).
\) %\end{equation*}
It follows that
\begin{equation}
\label{eq:upperbd}
    V(t_0,x)=\sup_{t_0 \leq \tau \leq 1}\mathbb{E}_{t_0,x}[X_{\tau}]\leq V^{*}(t_0,x).
\end{equation}

Now we show the reverse inequality $V(t_0,x) \geq V^*(t_0,x)$. If $X_{t_0} \geq c(t_0)$ then taking $\tau = t_0$ we have
$V(t_0,x) \geq x = V^*(t_0,x)$. It remains to consider the case $X_{t_0} < c(t_0)$.
Set $\tau^{*} = \inf\{s\geq t_0: X_{s}\geq c(s)\}$. It follows that
$X_{t\wedge\tau^{*}}\leq c(t)$ almost surely. Then, from \eqref{supermartingale_eq} we conclude that $V^{*}(t
\wedge \tau^*,
X_{t\wedge\tau^{*}})$ is a martingale. Using the Optional Stopping theorem again, we have
$$V^{*}(t_0,x)=\mathbb{E}_{t_0,x}[V^{*}(\tau^{*},X_{\tau^{*}})]=\mathbb{E}_{t,x}[X_{\tau^{*}}]$$
where the second equality comes from the fact that $V^{*}(t,c(t))=c(t)$.
Hence, 
\begin{equation*}
    V(t_0,x)=\sup_{t_0 \leq \tau \leq 1}\mathbb{E}_{t_0,x}[X_{\tau}]\geq \mathbb{E}_{t_0,x}[X_{\tau^{*}}]= V^{*}(t_0,x) 
\end{equation*}
Combining this result with \eqref{eq:upperbd} we have  $V(t_0,x) = V^*(t_0,x)$ and $\tau^*$ is optimal.
\end{proof}

\section{Optimal stopping of an $\alpha-$dimensional Bessel bridge}
A Brownian excursion can also be viewed as a 3-dimensional Bessel bridge. In this section
we are interested in extending the problem of Section~\ref{sec:BE} to allow for an $\alpha-$dimensional Bessel bridge as the
underlying stochastic process and to include more general objective functions of
power law form.

At least when $\alpha$ is an integer, the $\alpha$-dimensional Bessel process corresponds to the radial part of $\alpha$-dimensional Brownian motion. 
For $\alpha\geq 2$, $\alpha$-dimensional Brownian motion started at the origin never returns to zero. Correspondingly, an $\alpha$-dimensional Bessel process started at the origin never hits zero after time zero (provided $\alpha \geq 2$). 
Similarly, when $\alpha\geq 2$ an $\alpha$-dimensional Bessel bridge is strictly positive on time-interval $(0,1)$. However, for $\alpha \in (0,2)$, both an $\alpha$-dimensional Bessel process and an $\alpha$-dimensional Bessel bridge return to zero countably many times in $(0,1)$, as well as at time 1 for the Bessel bridge. This can cause complications in the study of the Bessel process, 
for example, when $0 < \alpha < 1$ the $\alpha$-dimensional Bessel process does not have a representation via a stochastic differential equation. Instead we work with the squared-Bessel process, and the squared Bessel bridge. Bessel processes and squared Bessel processes are discussed in detail in Revuz and Yor~\cite{RevuzYor}, see Chapter 11, and especially Exercise 3.11.

\begin{definition}
\label{def:BesQ}
    Suppose $\alpha>0$.
    A non-negative continuous process $Q=\{Q_{t}\}_{0\leq t\leq 1}$ is an $\alpha-$dimensional squared Bessel bridge
    process if it satisfies
    the SDE,
    \begin{equation*}
        dQ_{t}=\left(\alpha - \frac{2Q_{t}}{1-t} \right)dt+ 2 \sqrt{Q}_t dW_{t}, \qquad   0<t<1,
    \end{equation*}
    subject to $Q_0=0$.
\end{definition}

\begin{remark}
Given a process $Q$ as in Definition~\ref{def:BesQ} we can define an $\alpha$-dimensional Bessel bridge $X= (X_t)_{0 \leq t \leq 1}$ by $X_t = \sqrt{Q_t}$. If $\alpha \geq 2$ so that $Q>0$ on $(0,1)$ we can apply It\^{o}'s formula to deduce that $X$ solves the SDE $dX_t = \left( \frac{\alpha-1}{2X_t} - \frac{X_t}{1-t} \right) dt + dW_t$ and when $\alpha=3$ we recognise the Brownian excursion of Definition~\ref{def_BE}. When $\alpha \geq 2$ we can continue to work with this representation, but when $\alpha \in [1,2)$ we need to deal with the times when $X=0$, and when $\alpha \in (0,1)$, $X$ cannot be written as the solution of an SDE, hence our re-specification of the problem via the squared Bessel process.
\end{remark}

Our goal is to investigate the optimal stopping problem for the $\alpha$-dimensional Bessel bridge with power-law payoff, $\phi(x)=x^n$. 
For $\alpha \in (0,\infty)$ and $n \in (0,\infty)$ define $V(t,x)=\sup_{t\leq\tau\leq 1}\mathbb{E}_{t,x}[X_{\tau}^{n}]$. Then $V(t,x)=U(t,x^2)$ where 
\begin{equation}\label{value_general}
    U(t,q)=\sup_{t\leq\tau\leq 1}\mathbb{E}_{t,q}[Q_{\tau}^{n/2}],
\end{equation}
where $Q$ is the squared Bessel bridge introduced in Definition~\ref{def:BesQ}.
As in the last section, we aim to find the optimal stopping time $\tau^{*}$ in which the supremum in \eqref{value_general} is attained.
For the problem started at time $t_0 \in (0,1)$ we expect the optimal stopping rule to be
of threshold form:
\begin{equation*}
    \tau^{*}=\inf\{s\geq t_0: X_{s}\geq \sqrt{z(s)} \} = \inf\{s\geq t_0: Q_{s}\geq z(s)\},
\end{equation*}
where $z(\cdot)$ is the optimal boundary (for the squared Bessel bridge).
Following the Brownian excursion case,
in the continuation region  
$\mathcal{C}_Q =\{ (t,q) \subseteq (0,1) \times [0,z(t)) \}$ we have
\begin{equation*}
        dU(t, Q_{t})=\left[U_{t}(t,Q_{t}) +U_{q}(t,U_{t})\left(\alpha-\frac{2Q_{t}}{1-t}\right)
        +2 Q_t U_{qq}(t, Q_{t})\right]dt+ 2 \sqrt{Q_t} U_{q}(t,X_{t})dW_{t},
\end{equation*}
and on ${\mathcal C}_Q$ 
we expect that $U$ solves the pde
\begin{equation}
    U_{t}(t,q) +\left(\alpha-\frac{2q}{1-t}\right)U_{q}(t,q) + 2q U_{qq}(t, q) = 0
    \label{PDE_general}
\end{equation}
subject to $U(t,z(t))= z(t)^{n/2}$
and $U_{q}(t,z(t))= \frac{n}{2} z(t)^{\frac{n}{2}-1}$. 

We make the ansatz $z(t)=Z(1-t)$ and $V(t,x)=(1-t)^{\frac{n}{2}}f\left(\frac{x}{\sqrt{1-t}}\right)$ or equivalently 
\begin{equation}
    \label{eq:Udef}
U(t,q) = (1-t)^{\frac{n}{2}}g\left(\frac{q}{1-t}\right)
\end{equation}
for $Z$ a
constant and $f,g$ functions to be
determined such that $f(x)=g(x^2)$. On $[0,\sqrt{Z}]$ we have that $f$ solves the ODE
\( yf''(y)+(\alpha-1-y^{2})f'(y)-nyf(y)=0, \) 
subject to 
\( f(\sqrt{Z}) = Z^{n/2}, 
    f'(\sqrt{Z})=nZ^{(n-1)/2},\) 
which is equivalent to
\begin{equation}
    4yg''(y)+ 2(\alpha-y)g'(y)-ng(y)=0, \label{ode_ge}
\end{equation}
subject to $g(Z)  = Z^{n/2}$ and $g'(Z)  = \frac{n}{2}Z^{\frac{n}{2}-1}$.

In general there does not seem to be a closed form solution to \eqref{ode_ge} so we look for a solution in the
form of a power series
expansion. However, as well as the Brownian excursion with linear payoff $(\alpha=3,n=1)$ there are other cases where we have explicit solutions and before moving on to the general case we record those here.

Motivated by the Brownian excursion example, consider $G(y) = y^\theta$. Then $G$ solves \eqref{ode_ge} provided 
$2\theta(\theta-1) + \alpha \theta = 0$ and $-2 \theta - n=0$. Ruling out the case $\theta=0$ we find that $G$ is a solution provided $n = - 2 \theta = \alpha-2$, so that there is an explicit solution if $n=\alpha-2$ and then the solution is given by $G(y) = y^{-\frac{n}{2}}$. This solution explodes at 0, but we can find another solution $H$ of the form
\[ H(y) = y^{-\frac{n}{2}} \int_0^y e^{\frac{s}{2}} s^{\frac{n}{2}-1} \frac{ds}{2} = y^{1-\frac{\alpha}{2}} \int_0^y e^{\frac{s}{2}} s^{\frac{\alpha}{2}-2} \frac{ds}{2};    \qquad n = \alpha - 2 \]
for which $H$ and $H'$ are strictly positive and bounded at zero. Choosing an appropriate multiple of $H$ we can construct a solution to \eqref{ode_ge} which can be used as the basis for the candidate solution of the optimal stopping problem.

Another example from the literature where an explicit solution is known is the case of a reflecting Brownian bridge ($\alpha=1,n=1$), 
see Ekstr\"om and Wanntorp~\cite {Ekstrom_Wanntorp}. In this case, $G(y)= e^{{y}/{2}}$ solves \eqref{ode_ge}. Motivated by this example, consider $G(y) = y^\theta e^{y/2}$. We find that $G$ solves \eqref{ode_ge} provided 
either $(\theta=0,\alpha=n)$ or $(n=2,\alpha=2 - 2\theta)$. In particular, when $\alpha=n$ we find that the candidate value function is given by $U(t,q)=(1-t)^{n/2} n^{n/2}e^{\frac{(q-\sqrt{n})}{2(1-t)}}$ and $Z=\sqrt{n}$. (For the Brownian bridge we recover that the candidate value function is 
$V^*(t,x) = (1-t)^{1/2}e^{\frac{x^2-1}{2(1-t)}}$ and the candidate optimal stopping rule is $\tau^* = \inf \{ t \geq t_0 : X_t \geq \sqrt{1-t} \}$.)
Further, when $n=2$ we find $G(y)= y^{1 - \frac{\alpha}{2}} e^{y/2}$ is a solution to \eqref{ode_ge} and the candidate value function (at least when $\alpha>2$) is built from the second solution to \eqref{ode_ge}
\[ H(y) = E y^{1- \alpha/2} e^{y/2} \int_0^y e^{-v/2} v^{-2 + \frac{\alpha}{2}} \frac{dv}{2} \]
for some constant $E>0$.

Returning to the general case we look for a solution in power series form.
Let $\psi(y)=\sum_{k=0}^{\infty}A_{k}y^{k}$ where $A_0=1$ so that $\psi(0)=1$. Then we have $\psi'(y)=\sum_{k=1}^{\infty}kA_{k}y^{k-1}$ and $\psi''(y)=\sum_{k=2}^{\infty}k(k-1)A_{k}y^{k-2}$. Substituting these back into \eqref{ode_ge}, we obtain a recursive equation for the coefficients $(A_k)_{k \geq 0}$:
\begin{equation}
A_{k+1}=\frac{(2k+n)}{2(k+1)(2k+\alpha)} A_k; \qquad A_0=1, \label{powerseries_cond_1}
\end{equation}
which has solution
\begin{equation}
\label{eq:defCk}
   A_{k}=\frac{\Gamma(\frac{\alpha}{2})}{\Gamma(\frac{n}{2})}\frac{\Gamma(k+\frac{n}{2})}{\Gamma(k+\frac{\alpha}{2})2^{k}k!}. 
\end{equation}
Given $\psi$ we can define a second solution $\psi_2$ to  \eqref{ode_ge} of the form 
\[ \psi_2(y) = \psi(y) \int_0^q \frac{ s^{-\alpha/2} e^{s/2} }{\psi(s)^{2}} \frac{ds}{2} \]
and the general solution to \eqref{ode_ge} is
\begin{equation}
\label{eq:generalg}
g(y) = E_1 \psi(y) + E_2 \psi_2(y). 
\end{equation}
Note that $\psi_2(y) \sim y^{1- \alpha/2}$ for small $y$. 
We want a solution for which $g'$ is bounded on $[0,Z]$ and $g''$ is finite at 0 and for which the boundary conditions hold at $Z$. Together, the first set of conditions implies that $E_2=0$ and the boundary conditions at $Z$ imply that $Z$ solves $Z \psi'(Z) = \frac{n}{2} \psi(Z)$.  
This is equivalent to $F_{\alpha,n} (Z)=0$ where
\begin{equation}\label{D_solu}
   F_{\alpha,n}(z) = \sum_{k=0}^{\infty} (2k-n)A_{k}z^{k}.
\end{equation}

The next result follows from Lemma~\ref{lem:Appuniqueroot} in the Appendix.
\begin{corollary}
\label{cor:B}
    $F_{\alpha,n}$ has a unique positive root $Z = Z_{\alpha,n}$. Moreover, $F_{\alpha,n}(z) < 0$ on $[0,Z_{\alpha,n})$.
\end{corollary}

Recall that the root $Z_{\alpha,n}$ defines the optimal boundary to the optimal stopping problem for the $\alpha$-dimensional squared Bessel process with payoff $\phi(q) = q^{n/2}$ via $\tau^* = \inf \{ t \in (0,1) : Q_t \geq  Z(1-t) \}.$  
Moreover, the coefficient $E_1$ in \eqref{eq:generalg} can be computed using $Z^{n/2} = g(Z) = E_1 \psi(Z)$ to give
\begin{equation}
\label{eq:E1def}
    E_1 = \frac{Z_{\alpha,n}^{n/2}}{\sum_{k=0}^{\infty}A_{k}Z_{\alpha,n}^{k}}.
\end{equation}
Therefore, we can define a candidate function using \eqref{eq:Udef} and \eqref{eq:generalg} of the form
\begin{equation}\label{candidate_value_func_ge}
    U^{*}(t,q)=
    \begin{cases}
       \left(\frac{Z_{\alpha,n}^{n/2}}{\sum_{k=0}^{\infty}A_{k}Z_{\alpha,n}^{k}}\right)(1-t)^{\frac{n}{2}}\sum_{k=0}^{\infty}A_{k}\left(\frac{q}{{1-t}}\right)^{k},\quad
       &\text{if } q < Z_{\alpha,n} (1-t)\\
       q^{n/2},\quad &\text{if } q \geq Z_{\alpha,n} (1-t) ,
    \end{cases}
\end{equation}
where $(A_{k})_{k \geq 0}$ are given by
\eqref{eq:defCk} 
and $Z_{\alpha,n}$ is the unique solution to \eqref{D_solu}.

Recall that $g(y) = E_1 \psi(y) = E_1 \sum_{k \geq 0} A_k y^k$ on $[0,Z]$.
\begin{lemma}
\label{lem:glowerbd}
For $z \in [0,Z]$ we have $g(z) \geq z^{n/2}$.
\end{lemma}
\begin{proof} 
We have that $2y g'(y) - ng(y) = \frac{E_1}{2}F_{\alpha,n}(y) < 0$ on $[0,Z)$. Then, for $z \in (0,Z)$ we have $\ln \frac{g(Z)}{g(z)} = \int_z^Z \frac{g'(y)}{g(y)} dy < \frac{n}{2} \int_z^Z \frac{1}{y}dy = \ln \frac{Z^{n/2}}{z^{n/2}}$ and then using $g(Z) = Z^{n/2}$ we find $g(z) \geq z^{n/2}$ on $[0,Z]$. 
\end{proof}

The next result is important for the main theorem. However, the proof is long and intricate and so it is postponed
to the appendix. Note
that when $\alpha+n \leq 2$ there is nothing to prove. %in the last inequality.

\begin{prop}
\label{prop:Dinequality}
    For every $\alpha > 0$ and $n>0$ the root  $Z_{\alpha,n} \in (0,\infty)$ to
    $F_{\alpha,n}(\cdot)=0$ is such that %Moreover, $F_{\alpha,n}(d) < 0$ on $[0,D_{\alpha,n})$ and the root satisfies
    $Z_{\alpha,n}\geq \frac{\alpha+n-2}{2}$. %Moreover, $F_{\alpha,n}(d) < 0$ on $[0,B_{\alpha,n})$.
\end{prop}

\begin{theorem}
    \label{thm:BesQ}
    The candidate value function $U^{*}(t,q)$ given in \eqref{candidate_value_func_ge} coincides with the value
    function defined in
    \eqref{value_general}.
    Moreover, the associated stopping time $\tau^*$ given by
    \begin{equation*}
        \tau^{*}=\inf\{s\geq t: Q_{s}\geq z(s) \},
    \end{equation*}
    where $z(s)=Z_{\alpha,n}(1-s)$ and $Z_{\alpha,n}$ is the root given in Corollary~\ref{cor:B}, is 
    optimal.
\end{theorem}

\begin{proof}
The proof follows the proof of Theorem~\ref{thm:BE} and here we only describe the necessary changes.

Fix $t_0$ and for $t \geq t_0$ let $M_{t}=U^{*}(t,Q_{t})$. 
Recall that we chose a power series expansion such that $g$, $g'$, and $g''$ are all well defined at 0.
Hence $U^*_q$ is bounded on $[0,Z]$ and $U^*_{qq}$ exists on $[0,Z)$ so that we can apply It\^{o}'s formula.
We have
\begin{equation*}
    dM_{t}= h(t,Q_t)\mathbb{I}\{Q_{t}\geq z(t) \}dt +U_{q}^{*}(t,Q_{t}) 2 \sqrt{Q_t} dW_{t}
\end{equation*}
where
\begin{equation*}
  h(t,q)= n q^{\frac{n}{2}-1}\left( \frac{\alpha + n - 2}{2} -  \frac{q}{1-t}\right).
\end{equation*}
Since $Z_{\alpha,n} \geq \frac{\alpha+n-2}{2}$ we have that $h \leq 0$ on the stopping region $Q_t \geq
Z_{\alpha,n} (1-t)$.
It follows that $A = \{ A_t \}_{t_0 \leq t \leq 1}$ given by $A_{t}:=\int_{t_0}^{t}h(s,Q_{s})\mathbb{I}\{Q_{s}\geq
z(s)\}ds$ is a decreasing process, null at $t_0$.

Define the local martingale $N= \{ N_t \}_{t_0 \leq t \leq 1}$ by $N_{t}=\int_{t_0}^{t}U^{*}(s,Q_{s}) 2 \sqrt{Q_s} dW_{s}$. Then $M_t = U^*(t_0,Q_{t_0}) + A_t +
N_t$. Since $M \geq 0$ we have $N \geq -U^*(t_0,Q_{t_0})-A_t \geq -U^*(t_0,Q_{t_0})$ and therefore the local martingale $N$ is bounded
below. Hence $N$ is a supermartingale.
It follows that
$\mathbb{E}[M_{t}|\mathcal{F}_{s}]=U^*(t_0,Q_{t_0}) + \mathbb{E}[A_{t}+N_{t}|\mathcal{F}_{s}] \leq U^*(t_0,Q_{t_0}) + A_{s}+ N_{s}=M_{s}$
so that $M$ is a supermartingale. By Lemma~\ref{lem:glowerbd}  $g(z) \geq z^{n/2}$ on $[0,Z]$ and hence $U^*(t,Q_t) \geq Q_t^{n/2}$ on $[0,\infty)$ and it follows that
for any stopping time $\tau$ with $\tau \geq t_0$,
\[ \mathbb{E}_{t_0,x}[Q_\tau^{n/2}] \leq \mathbb{E}_{t_0,x} [U^*(\tau,Q_\tau)] \leq M_{t_0} = U^*(t_0,Q_{t_0}) .\]
Thus $U(t_0,q) \leq U^*(t_0,q)$.

The remaining details which show that $M$ is a martingale for $t_0 \leq t \leq \tau^*$ and hence that $U(t_0,q) \geq U^*(t_0,q)$ follow exactly as in Theorem~\ref{thm:BE} using that $\sqrt{y}g'(y)$ is bounded on $[0,Z]$ for our chosen $g$. 

\end{proof}

\appendix
\section{Proofs of Key Results}

\begin{proof}[Proof of Lemma~\ref{lemma_f}]
Recall that $f$ is given by $f(y)= \frac{B}{y} \int_0^y e^{\frac{t^2}{2}} dt$. By l'H\^opital's rule we have $f(0)=B$. Also $f'(y)= \frac{1}{y} (Be^{y^2/2} - f(y))$ and again by l'H\^opital's rule $f'(0)= - f'(0)$ so that $f'(0)=0$.

Also, $f$ is the solution to $yf''(y)+(2-y^{2})f'(y)-yf(y)=0$ in $[0,C]$ with $f(C)=C$ and $f'(C)=1$. Then, in $[0,C]$ we have
 \begin{equation}
 \label{eq:f''}
     f''(y)=\frac{yf(y)-(2-y^{2})f'(y)}{y}.
 \end{equation}
 For $y>0$ we have $\left(1-\frac{y^{2}}{2}-\frac{y^{4}}{2}\right)e^{\frac{y^{2}}{2}}<e^{\frac{y^{2}}{2}}$.
 Integrating we find $$\frac{y(2-y^{2})}{2}e^{\frac{y^{2}}{2}} =\int_0^y \left( 1 - \frac{t^2}{2} - \frac{t^4}{2} \right)
e^{\frac{t^{2}}{2}} dt <\int_{0}^{y}e^{\frac{t^{2}}{2}}dt.$$ Then, since $f(y)= \frac{B}{y} \int_0^y e^{\frac{t^2}{2}} dt$, we have $B(2-y^2) e^{\frac{y^2}{2}} < 2 f(y)$.
Hence $(2-y^2)f'(y) = (2-y^2) \left( \frac{B}{y} e^{\frac{y^2}{2}} - \frac{f(y)}{y} \right) < 2 \frac{f(y)}{y} - (2-y^2) \frac{f(y)}{y} = y f(y)$. Substituting this into \eqref{eq:f''} we find that $f''>0$ or equivalently $f$ is convex on $[0,C]$.  

The fact that $f(y) \geq y$ then follows from the fact that $f$ is convex and hence $f'(y) \leq f'(C)=1$ on $[0,C]$ together with $f(C)=C$.
\end{proof}

\begin{lemma}
\label{lem:Appuniqueroot}
Let $h:[0,\infty)\to \mathbb{R}$ be given by $h(y)= \sum_{j \geq 0} a_j b_j y^j$ where $(a_j)_{0 \leq j < \infty}$ and $(b_j)_{0 \leq j < \infty}$ are such that $a_j>0$ and $b_{j+1} \geq b_j$ with $b_0<0$ and $b_J>0$ for some $J<\infty$. Moreover, suppose $h$ has an infinite radius of convergence.

Then $h$ has a unique positive root $y_*$ in $(0,\infty)$ and $h(y)<0$ if and only if $y \in [0,y_*)$.
\end{lemma}

\begin{proof}
Since $h(y)$ is a power series, it is continuous and differentiable within the radius of convergence, for details, see \cite[9.4.12 Differentiation Theorem]{Bartle_Sherbert_2011}. 
Note that $h(0)=b_{0}a_{0}<0$, and $\lim_{y\to\infty} h(y) > \lim_{y \rightarrow \infty} \sum_{j=0}^{J} a_j b_{j} y^j =\infty$.  Hence, by the Intermediate Value Theorem, there exists at least one $y_{*}$ such that $h(y_{*})=0$. 

Now we want to show the uniqueness. Let $y_*$ be a root of $h$. Since $(b_{j})_{j \geq 1}$ is monotonically increasing with $b_{0}<0$ and $b_J>0$ there exists $J_0>0$ such that for every $j<J_0$ we have $b_{j}<0$, and for every $j\geq J_0$ we have $b_{j}\geq 0$. 
Then, for $y<y_*$, 
\begin{eqnarray*}
    h(y)&=&\sum_{j=0}^{J_0-1}a_{j}b_{j}\left(\frac{y}{y_*}\right)^{j}y_*^{j}+ \sum_{j=J_0}^{\infty}a_{j}b_{j}\left(\frac{y}{y_*}\right)^{j}y_*^{j}\\
    &=&\bigg(\frac{y}{y_*}\bigg)^{J_0}\bigg[\sum_{j=0}^{J_0-1}a_{j}b_{j}\bigg(\frac{y}{y_*}\bigg)^{j-J_0}y_{*}^{j}+\sum_{j=J_0}^{\infty}a_{j}b_{j}\bigg(\frac{y}{y_{*}}\bigg)^{j-J_0}y_*^{j}\bigg]\\
    &<&\bigg(\frac{y}{y_{*}}\bigg)^{J_0}\left[\sum_{j=0}^{J_0-1}a_{j}b_{j}y_{*}^{j}+\sum_{j=J_0}^{\infty}a_{j}b_{j}y_{*}^{j}\right]\\
    &=& 0.
\end{eqnarray*}
Hence, the root $y_*$ of $h$ must be unique and $h<0$ on $[0,y_*)$. 

\end{proof}

\begin{proof}[Proof of Corollary~\ref{cor:B}]
Let $(a_j)_{j \geq 1}$ and $(b_j)_{j \geq 1}$ be given by $a_j = A_j$ and $b_j = (2j-n)$. Then $a_j>0$ for all $j$ and $(b_j)_{j \geq 0}$ is increasing with $b_0<0$ and $b_J>0$ for $J>n/2$.

Since $\lim_{k\to\infty}\frac{[2(k+1)-n]A_{k+1}}{(2k-n)A_{k}}=0$, we conclude that $F_{\alpha,n}$ has an infinite radius of convergence, for details see \cite[p. 269-270]{Bartle_Sherbert_2011}. 
Then by Lemma \ref{lem:Appuniqueroot} $F_{\alpha,n}$ has a unique root. Moreover, from the derivation of the optimal boundary, we know $Z_{\alpha,n}$ is a solution to $y\psi'(y)=\frac{n}{2}\psi(y)$, and thus $Z_{\alpha,n}$ is the unique root of $F_{\alpha,n}$.
\end{proof}

Given the significance of Proposition~\ref{prop:Dinequality} we restate it here.
\begin{prop}
\label{prop:DinequalityA}
    For every $\alpha > 0$ and $n>0$ the root  $Z_{\alpha,n} \in (0,\infty)$ to
    $F_{\alpha,n}(\cdot)=0$ is such that %Moreover, $F_{\alpha,n}(d) < 0$ on $[0,D_{\alpha,n})$ and the root satisfies
    $Z_{\alpha,n}\geq \frac{\alpha+n-2}{2}$. %Moreover, $F_{\alpha,n}(d) < 0$ on $[0,B_{\alpha,n})$.
\end{prop}

\begin{proof}
Recall that 
$$F_{\alpha,n}(z)= \sum_{k\geq 0}(2k-n)A_{k}z^{k} \quad \text{where} \quad A_{k} =\frac{\Gamma(\frac{\alpha}{2})}{\Gamma(\frac{n}{2})}\frac{\Gamma(k+\frac{n}{2})}{\Gamma(k+\frac{\alpha}{2})2^{k}k!}. $$

It follows from Corollary~\ref{cor:B} that $F_{\alpha,n}$ has a unique positive root $Z_{\alpha,n}$ on $(0,\infty).$ Moreover, again by Corollary~\ref{cor:B}, $Z_{\alpha,n}>\frac{\alpha+n}{2}-1$ if and only if either $\frac{\alpha+n}{2}-1 \leq 0$ or $F_{\alpha,n}(\frac{\alpha+n}{2}-1)<0.$

If $\frac{\alpha + n}{2} - 1 \leq 0$ there is nothing to prove. 
So, suppose $\frac{\alpha + n}{2} - 1 > 0$

Write $\alpha = n + 2 + 2 \gamma$, where $\gamma>-2$.  Then $\frac{\alpha + n}{2} - 1 = n + \gamma > 0$. We have 
\[ F_{\alpha,n}(z) = \frac{\Gamma( \frac{n}{2} + 1 + \gamma)}{\Gamma(\frac{n}{2})} \sum_{k=0}^\infty \frac{(2k-n)}{2^k k!} \frac{\Gamma(k+\frac{n}{2})}{\Gamma(k+\frac{n}{2}+ 1 + \gamma)} z^k, \]
Let $H = \frac{\Gamma(\frac{n}{2})}{\Gamma(\frac{n}{2}+1+ \gamma)}F_{\alpha,n}(n+\gamma)$. It is sufficient to show that $H<0$. The idea is to find an equivalent infinite sum expression for $H$ in which all the terms in the sum are non-positive.

For constants $D$, $B$ and $\Delta$ with $\Delta>0$  define $\Lambda = \Lambda_{D,B,\Delta}$ by 
\[ \Lambda_{D,B,\Delta} = \sum_{k=0}^\infty \frac{(n+\gamma)^k}{2^k k!} \frac{\Gamma(k+\frac{n}{2})}{\Gamma(k+\frac{n}{2}+\Delta+1+\gamma)} \left[ \frac{D}{2^{\Delta-1}}k + \frac{B}{2^{\Delta}}n \right].  \] 
Note that a sufficient condition for $\Lambda<0$ is both $D<0$ and $B<0$.
Then,
\begin{eqnarray*}
\lefteqn{\Lambda_{D,B,\Delta}} \\
&= & \sum_{k=1}^\infty \frac{D}{2^{\Delta-1}}\frac{k}{2^k k!} \frac{\Gamma(k+\frac{n}{2})}{\Gamma(k+\frac{n}{2}+\Delta+1+\gamma)} (n + \gamma)^k \\
& &\hspace{+3cm}+ \sum_{k=0}^\infty \frac{B}{2^{\Delta}}\frac{n}{2^k k!} \frac{\Gamma(k+\frac{n}{2})}{\Gamma(k+\frac{n}{2}+\Delta+1+\gamma)} (n + \gamma)^k \\
&=& \sum_{j=0, j=k-1}^\infty  \frac{D}{2^{\Delta}} \frac{1}{2^j j!} \frac{\Gamma(j+\frac{n}{2}+1)}{\Gamma(j+\frac{n}{2}+ \Delta+2+\gamma)} (n + \gamma)^{j+1} \\
&&\hspace{+3cm}+ \sum_{j=0}^\infty \frac{B}{2^{\Delta}}\frac{n}{2^j j!} \frac{\Gamma(j+\frac{n}{2})}{\Gamma(j+\frac{n}{2}+ \Delta+1+\gamma)} (n + \gamma)^j \\
&= & \sum_{j=0}^\infty \frac{(n+\gamma)^j}{2^j j!} \frac{\Gamma(j+\frac{n}{2})}{\Gamma(j+\frac{n}{2}+ \Delta+2+\gamma)} \left[ \frac{D}{2^{\Delta}}(n + \gamma)(j+\frac{n}{2}) + \frac{B}{2^{\Delta}}n(j+\frac{n}{2}+ \Delta+1+\gamma)\right] \\
&=&  \Lambda_{\tilde{D},\tilde{B}, \Delta+1} 
\end{eqnarray*}
where
\begin{equation}
    \label{eq:tildeDB}
    \tilde{D}=(n+\gamma)D+Bn; \quad \mbox{and} \quad \tilde{B}=(n+\gamma)D+(n+2\Delta+2+2\gamma)B.
\end{equation}
Thus
\begin{equation}
\label{eq:iterate}
\Lambda_{D,B,\Delta} = \Lambda_{(n+\gamma)D+Bn,(n+\gamma)D+(n+2\Delta+2+2\gamma)B,\Delta+1} .
\end{equation}

We split the proof into three cases.

\noindent{\bf Case 1: $\alpha+n > 2$ and $|\alpha-n| \leq 2$.}
From the definition of $\Lambda$ we have that $H=\Lambda_{1,-1,0}$. Moreover,
from \eqref{eq:iterate} with $D=1$, $B=-1$ and $\Delta=0$ we deduce that $H=\Lambda_{\gamma,-(2+\gamma),1}$.
It follows that if $\gamma \leq 0$ and $-(2+ \gamma) \leq 0$ (equivalently $-2 \leq \gamma \leq 0$ or $|\alpha-n| \leq 2$) then %all the coefficients are non-positive and 
$H < 0$. 

\noindent{\bf Case 2: $\alpha+n > 2$ and $\alpha> n+2$.}
In this case $\gamma>0$, and we cannot immediately conclude from $H=\Lambda_{\gamma,-(2+\gamma),1}$ that $H<0$. Instead, we repeat the rewriting of $H$ in terms of equivalent expressions for $\Lambda$:
\[ H = \Lambda_{1,-1,0} = \Lambda_{\gamma,-(2+\gamma),1} = \ldots = \Lambda_{D_r,B_r,r} = \ldots \]
where $D_2=\gamma^{2}-2n$ and $B_2=-(\gamma^{2}+8\gamma+2n+8)$ and more generally from \eqref{eq:tildeDB}
\begin{eqnarray}
    D_{r+1}&=&(n+\gamma)D_{r}+nB_{r}\label{iteration_d}\\
    B_{r+1}&=&(n+\gamma)D_{r}+(n+2r+2+2\gamma)B_{r}\label{iteration_b}.
\end{eqnarray}
We argue that for any $\gamma>0$ we can find $r$ large enough so that both $D_r <0$ and $B_r<0$. Then we can conclude that $H<0$.

Let $\sP_r$ be given by the statement
\[
\sP_r: \quad    D_{r}\leq\gamma^{r}, \quad \mbox{and} \quad B_{r}\leq-\gamma^{r}\left(1+\frac{2r}{\gamma}\right).
\]
It is immediate that $\sP_0$ is true. We prove that $\sP_r$ is true for all $r$ by induction.

Suppose that $\sP_k$ is true so that $D_k \leq \gamma^k$ and $B_k \leq -\gamma^k  \left( 1 + \frac{2k}{\gamma} \right)$. Then,
\begin{equation}
\label{eq:Diteration}
    D_{k+1}=(n+\gamma)D_{k}+nB_{k}
    \leq(n+\gamma)\gamma^{k}-\gamma^{k}\left( 1 + \frac{2k}{\gamma} \right)n
    =\gamma^{k+1} \left( 1 -\frac{2kn}{\gamma^2} \right)
    \leq \gamma^{k+1}
\end{equation}
and
\begin{eqnarray*}
    B_{k+1}&=&(n+\gamma)D_{k}+(n+2k+2+2\gamma)B_{k}\\
    &\leq&(n+\gamma)\gamma^{k}-(n+2k+2+2\gamma)\gamma^k  \left( 1 + \frac{2k}{\gamma} \right)\\
    &=&-\gamma^{k+1}-6k\gamma^{k}-2\gamma^{k}-\gamma^{k-1}(2kn+4k^{2}+4k)\\
    &\leq& %-\gamma^{k+1}-2k\gamma^{k}-2\gamma^{k}\\ &=&
    -\gamma^{k+1} \left( 1 + \frac{2(k+1)}{\gamma} \right).
\end{eqnarray*}
Hence $\sP_{k+1}$ is true, and by induction $\sP_r$ is true for all $r \geq 0$.

Now choose $r_0$ such that $r_0> \frac{\gamma^2}{2n}$. Then, from \eqref{eq:Diteration} we see that not only do we have that $D_{r_0+1} \leq \gamma^{r_0+1}$ but also from the penultimate expression that $D_{r_0+1} < 0$. Then
$H = \Lambda_{D_{r_0+1},B_{r_0+1},r_0+1} < 0$.

\noindent{\bf Case 3: $\alpha+n > 2$ and $\alpha < n-2$.}
    In this case we let $n =\alpha+2+2\delta$, where $\delta>0$. 
    This is a reparameterization of the problem in terms of $\alpha$ and $\delta$ (rather than $n$ and $\gamma$) where
    \begin{equation*}
        n+\gamma=\alpha+\delta, \quad \quad n=\alpha+2+2\delta. 
    \end{equation*}
    $H$ and $\Lambda$ can also be expressed in terms of $\alpha$ and $\delta$. Thinking of $\alpha>0$ as fixed and writing $(D_{\alpha,r},B_{\alpha,r})$ rather than $(D_r,B_r)$
    we find that 
\[ H = \Lambda_{1,-1,0} = \ldots \Lambda_{D_{\alpha,r},B_{\alpha,r},r} = \ldots \]
where now the iteration formulae in \eqref{iteration_d} and \eqref{iteration_b} are replaced by 
    \begin{eqnarray}
       D_{\alpha,r+1}&=&D_{\alpha,r}(\alpha+\delta)+B_{\alpha,r}(\alpha+2+2\delta)\label{iteration_D_Tilde}\\
       B_{\alpha,r+1}&=&D_{\alpha,r}(\alpha+\delta)+B_{\alpha,r}(2r+\alpha)\label{iteration_B_Tilde}. 
    \end{eqnarray}
subject to $D_{\alpha,0}=1$, $B_{\alpha,0}=-1$.
    The goal here is again to show that for each $\delta>0$ there exists $r\geq 1$, such that both $D_{\alpha,r}<0$ and $B_{\alpha,r}<0$, for if so $H<0$.   

First consider the case where $\alpha=0$. Then, \eqref{iteration_D_Tilde} and \eqref{iteration_B_Tilde} simplify to
\[ D_{0,r+1}=\delta D_{0,r}+2(1+\delta) B_{0,r}; \hspace{2cm} B_{0,r+1}=\delta D_{0,r}+2r B_{0,r}, \]
where $D_{0,0}=1$ and $B_{0,0}=-1$. Then
$D_{0,1}=-(\delta + 2)$ and $B_{0,1}= \delta$
and we can easily compute the next six iterations,
which are recorded in Table~\ref{tab:narrow_d_b}.  
\begin{table}[h!]
\centering
\resizebox{\textwidth}{!}{%
\begin{tabular}{|c|r|r|}
\hline
$j$ & $D_{0,j}$ \hspace{3cm} & $B_{0,j}$ \hspace{3cm} \\
\hline
2 & $\delta^{2}$ & $-\delta^{2}$ \\
3 & $-\delta^{3} - 2\delta^{2}$ & $\delta^{3} - 4\delta^{2}$ \\
4 & $\delta^{4} - 8\delta^{3} - 8\delta^{2}$ & $-\delta^{4} + 4\delta^{3} - 24\delta^{2}$ \\
5 & $-\delta^{5} - 2\delta^{4} - 48\delta^{3} - 48\delta^{2}$ & $\delta^{5} - 16\delta^{4} + 24\delta^{3} - 192\delta^{2}$ \\
6 & $\delta^{6} - 32\delta^{5} - 32\delta^{4} - 384\delta^{3} - 384\delta^{2}$ & $-\delta^{6} + 8\delta^{5} - 208\delta^{4} + 192\delta^{3} - 1920\delta^{2}$ \\
7 & $-\delta^{7} - 18\delta^{6} - 432\delta^{5} - 416\delta^{4} - 3840\delta^{3} - 3840\delta^{2}$ & $\delta^{7} - 44\delta^{6} + 64\delta^{5} - 2880\delta^{4} + 1920\delta^{3} - 23040\delta^{2}$ \\
\hline
\end{tabular}
}
\caption{Expressions for $D_{0,j}$ and $B_{0,j}$ for $2\leq j\leq 7$}
\label{tab:narrow_d_b}
\end{table}

Our goal is to find $r_0$ such that $D_{0,r_0} \leq 0$ and $B_{0,r_0} \leq 0$ (with at least one inequality being strict). Given the alternating nature of the leading terms, this is more subtle than in previous cases.

If $r$ is odd, let $\sQ_r$ be the statement
\[
\sQ_r: \quad    D_{0,r} \leq -\delta^{r}-2r\delta^{r-1} \quad \mbox{ and } \quad B_{0,r} \leq \delta^{r}-2r\delta^{r-1} .
\]
If $r$ is even, 
let $\sQ_r$ be the statement
\[
\sQ_r: \quad    D_{0,r}\leq -\delta^{r}-4r\delta^{r-1} \quad \mbox{ and }  \quad B_{0,r} \leq \delta^{r} .
\]
We observe that $\sQ_7$ holds. We now show that if $\sQ_r$ holds for $r$ odd then $\sQ_{r+1}$ holds, and if $\sQ_r$ holds for $r$ even then $\sQ_{r+1}$ holds. It will then follow by induction that $\sQ_r$ holds for all $r \geq 7$.

Suppose $r \geq 3$ is odd and $\sQ_r$ holds. Then  
     \begin{eqnarray*}
         D_{0,r+1} = \delta D_{0,r}+2(1+\delta) B_{0,r}
            &\leq & \delta(-\delta^{r}-2r\delta^{r-1})+2(1+\delta)(\delta^{r}-2r\delta^{r-1})\\
            &= &\delta^{r+1}-(6r-2)\delta^{r}-4r\delta^{r-1}\\
            &\leq &\delta^{r+1}-4(r+1)\delta^{r},
         \end{eqnarray*}
where the last inequality holds since $2r \geq 6$.     Moreover, 
\begin{eqnarray*}
            B_{0,r+1} =  \delta D_{0,r}+2r B_{0,r}
           &\leq& \delta(-\delta^{r}-2r\delta^{r-1})+2r(\delta^{r}-2r\delta^{r-1})\\
           &=&-\delta^{r+1}-4r^{2}\delta^{r-1}\\
           &\leq& -\delta^{r+1}.
     \end{eqnarray*}
Hence $\sQ_{r+1}$ holds. Now suppose $r \geq 2$ is even and $\sQ_r$ holds.
Then
     \begin{eqnarray*}         
           D_{0,r+1} = \delta D_{0,r}+2(1+\delta) B_{0,r}
          &\leq & \delta(\delta^{r}-4r\delta^{r-1})-2(1+\delta)\delta^{r}\\
          &=&-\delta^{r+1}-(4r+2)\delta^{r}\\
          &\leq& -\delta^{r+1}-2(r+1)\delta^{r},
     \end{eqnarray*}
     and
\begin{equation*}             
B_{0,r+1}= \delta D_{0,r}+2r B_{0,r}
\leq \delta(\delta^{r}-4r\delta^{r-1})-2r\delta^{r}
            =\delta^{r+1}-6r\delta^{r}
            \leq\delta^{r+1}-2(r+1)\delta^{r}.
     \end{equation*}
We conclude that $\sQ_{r+1}$ holds.

It follows by induction that $\sQ_r$ holds for all $r \geq 7$.
Given $\delta>0$ choose $r^* \geq 7$ with $r$ odd such that 
$2r^*>\delta$. Then $\sQ_{r^*}$ holds. In particular, $D_{0,r^*}<0$ holds and, by closer inspection of the statement of $\sQ_r$, also $B_{0,r^*}<0$. Hence $H<0$.

Finally, we extend the result that for all $\delta>0$ there exists $r^*$ such that $\Lambda_{D_{0,r^*},B_{0,r^*},r^*} < 0$ from the case $\alpha=0$ to $\alpha>0$.
Let $\sR_r$ be the statement
\[ \sR_r:   \quad \quad D_{\alpha,r}\leq D_{0,r},\quad \text{and} \quad B_{\alpha,r}\leq B_{0,r} \]
where $D_{\alpha,r}$ and $B_{\alpha,r}$ are given by the iterations in \eqref{iteration_D_Tilde} and \eqref{iteration_B_Tilde}.
Note that $\sR_0$ is true since 
$ D_{\alpha,0}= D_{0,0}=1$ and $ B_{\alpha,0}= B_{0,0}=-1$. Note also that $D_{0,r}+B_{0,r} \leq 0$ for all $r \geq 0$, by inspection for $r \leq 7$ and by the fact that $\sQ_r$ is true for all $r \geq 7$.

 Assume that $\sR_r$ holds for $r=k$. Then
    \begin{eqnarray*}
         D_{\alpha,k+1}&=& D_{\alpha,k}(\alpha+\delta)+ B_{\alpha,k}(\alpha+2+2\delta)\\
        &\leq & D_{0,k}(\alpha+\delta)+ B_{0,k}(\alpha+2+2\delta)\\
        &= &\alpha( D_{0,k}+ B_{0,k})+\delta D_{0,k}+(2+2\delta) B_{0,k}\\
        &= &\alpha( D_{0,k}+ B_{0,k})+ D_{0,k+1}\\
        &\leq &D_{0,k+1}.
    \end{eqnarray*}
     Similarly, since $\sR_k$ holds, 
    \begin{eqnarray*}
       B_{\alpha,k+1}&=& D_{\alpha,k}(\alpha+\delta)+ B_{\alpha,k}(2k+\alpha)\\
      &\leq&  D_{0,k}(\alpha+\delta)+ B_{0,k}(2k+\alpha)\\
      &=&\alpha( D_{0,k}+ B_{0,k})+\delta D_{0,k}+2k B_{0,k}\\
      &=&\alpha( D_{0,k}+ B_{0,k})+ B_{0,k+1}\\
      &\leq& B_{0,k+1}.
    \end{eqnarray*}
Hence $\sR_{k+1}$ holds, and we conclude that $\sR_r$ holds for all $r$. Recall that for all $\delta>0$ there exists $r^*$ with $r^* \geq 7$ with $r^*$ odd such that $2 r^*> \delta$,  and then $ D_{0,r^{*}}<0$ and $ B_{0,r^{*}}<0$. Then we also have $ D_{\alpha,r^{*}}\leq D_{0,r^{*}}<0$ and $ B_{\alpha,r^{*}}\leq B_{0,r^{*}}<0$. Hence, $H= \Lambda_{D_{\alpha,r^*},B_{\alpha,r^*},r^*}<0$ as required.
\end{proof}

\end{document}